\documentclass[reqno, 12pt]{amsart}
\usepackage{appendix}
\usepackage[utf8]{inputenc}
\usepackage[T1]{fontenc}
\usepackage{lmodern}
\usepackage{amsmath}
\usepackage{amsthm}
\usepackage{amssymb}
\usepackage{latexsym}
\usepackage{amsmath,amsfonts,amssymb}
\usepackage{upgreek}
\usepackage{enumerate}
\usepackage{mathrsfs}
\usepackage{mathscinet}
\usepackage{stmaryrd} 
\usepackage{graphicx}
\usepackage{xcolor}
\usepackage[all,cmtip]{xy}
\usepackage{url}
\usepackage{hyperref}
\usepackage[a4paper, hmargin=3cm, vmargin=3cm]{geometry}

\newcommand\mathens[1]{\mathbb{#1}} 

\newcommand{\ud}{\mathrm{d}}

\newcommand{\N}{\mathens{N}}
\newcommand{\Z}{\mathens{Z}}

\newcommand{\R}{\mathens{R}}

\newcommand{\T}{\mathens{T}}

\newcommand\sphere[1]{\mathens{S}^{#1}}

\newcommand{\cont}{\mathrm{Cont}}
\newcommand{\conto}{\mathrm{Cont}_{0}}

\newcommand{\id}{\mathrm{id}}
\newcommand{\nlmas}\upmu

\newtheorem{theo}{Theorem}[section] 

\newtheorem{thm}{Theorem}[section]
\newtheorem{lem}[thm]{Lemma}
\newtheorem{cor}[thm]{Corollary}
\newtheorem{prop}[thm]{Proposition}

\newtheorem{prop-def}[thm]{Definition-proposition}
\theoremstyle{plain}
\newtheorem*{coro*}{Corollary}

\theoremstyle{definition}

\theoremstyle{remark}

\newtheorem{rem}[thm]{Remark}

\newcommand\tpsi{\widetilde{\psi}}
\newcommand\tphi{\widetilde{\varphi}}
\newcommand\teta{\widetilde{\eta}}

\newcommand\Leg{\mathcal{L}}
\newcommand\uLeg{\widetilde{\Leg}}
\newcommand\Gcont{\mathcal{G}}
\newcommand\uGcont{\widetilde{\Gcont}}
\newcommand\cleq{\preceq}

\newcommand\tLambda{\widetilde{\Lambda}}

\newcommand\dspec{\mathrm{d}_\mathrm{spec}}

\newcommand\Nspec[1]{\left|#1\right|_\mathrm{spec}}

\newcommand\losc{\nu_\mathrm{osc}}
\newcommand\dCSosc{\mathrm{d}_\mathrm{CS,osc}}

\DeclareFontFamily{U}{mathb}{\hyphenchar\font45}
\DeclareFontShape{U}{mathb}{m}{n}{
      <5> <6> <7> <8> <9> <10> gen * mathb
      <10.95> mathb10 <12> <14.4> <17.28> <20.74> <24.88> mathb12
}{}
\DeclareSymbolFont{mathb}{U}{mathb}{m}{n}
\DeclareMathSymbol{\cll}{3}{mathb}{"CE}

\makeatletter\let\@wraptoccontribs\wraptoccontribs\makeatother

\title{Invariant distances on spaces of Legendrians}
\vspace{-2cm}
\author[P.-A. Arlove]{Pierre-Alexandre Arlove}
\address{P.-A. Arlove, Universit\'e de Strasbourg, IRMA UMR 7501, F-67000 Strasbourg, France}
\email{paarlove@unistra.fr}

\begin{document}

\begin{abstract}
We construct new unbounded invariant distances on the universal cover of certain Legendrian isotopy classes. This is the first instance where unboundedness of an invariant distance is obtained without assuming the existence of a positive loop of contactomorphisms. We also show that invariant distances on Legendrian isotopy classes have to be discrete. 
\end{abstract}

\maketitle


\section{Introduction}

\subsection{Main results}
Let $\Lambda_*$ be a closed and connected Legendrian lying inside a cooriented contact manifold $(M,\xi)$. We denote by $\Leg(\Lambda_*)$ the isotopy class of $\Lambda_*$, i.e. the Legendrians that are Legendrian isotopic to $\Lambda_*$, and by $\uLeg(\Lambda_*)$ its universal cover. The group of contactomorphisms $\cont(M,\xi)$, i.e. the diffeomorphisms of $M$ that preserve the contact distribution and its coorientation, contains the group $\cont_0(M,\xi)$ of contactomorphisms contact isotopic to the identity, which itself contains the group $\Gcont(M,\xi)$ of contactomorphisms compactly contact isotopic to the identity. Both groups $\cont_0(M,\xi)$ and $\Gcont(M,\xi)$ act transitively on $\Leg(\Lambda_*)$. The same holds for their universal covers $\uGcont(M,\xi)\subset \widetilde{\cont_0}(M,\xi)$  that both act transitively on $\uLeg(\Lambda_*)$. To simplify notation, we shall, when unambiguous, denote these spaces by $\Leg$, $\uLeg$, $\cont$, $\conto$, $\Gcont$ and $\uGcont$. We refer to Section \ref{se:preliminaire} and the Appendix for more details.\\

It follows from the above that $\uLeg$ (resp. $\Leg$) is an homogeneous space, thus a distance $\ud$ on $\uLeg$ (resp. on $\Leg$) is said to be invariant if $\ud(z\cdot y,z\cdot x)=\ud(x,y)$ for all $z\in\widetilde{\cont_0}$, and all $x,y\in\uLeg$ (resp. for all $z\in\conto$, and all $x,y\in\Leg$). Such distances have not been very much studied in this context (see \cite{allais2023spectral,discriminante}\footnote{non-invariant distances on $\Leg$ and $\uLeg$ have been studied in \cite{allais2023spectral,Nakamura2023,RosenZhang2020}}) while their analogues for contactomorphisms (\cite{Arlove2023,discriminante,courtemassot,FPR,sandonmetrique,shelukhin,Zapolsky}) or Lagrangians in symplectic geometry (\cite{BiranCornea2021,Chekanov2000,vincenthumiliere,ostroverlagrangian,Seyfaddini_2014,trifa2024hoferdistancelagrangianlinks,viterbo2024supportshumilierecompletiongammacoisotropic}) have generated a lot of interest. \\

As first observed by Eliashberg and Polterovich in \cite{EP00}, gaining a better geometric understanding of $\Gcont$, $\conto$, or $\Leg$ and their universal covers crucially depends on the existence or non-existence of positive loops in these spaces.
 Recall that an isotopy in $\cont$, $\Gcont$ or $\Leg$ is said to be positive if its speed is everywhere positively transverse to the cooriented contact distribution. Similarly, a path in $\widetilde{\conto}$, $\uGcont$ or $\uLeg$ is positive if it covers a positive isotopy of $\conto$,  $\Gcont$ or $\Leg$. 
See Section \ref{se:partial order} for more details. So far, the existence of unbounded invariant distances on $\uLeg$ (resp. $\uGcont$) has been proven under strong assumptions: $\uLeg$ (resp. $\uGcont$) is \textit{orderable}, i.e. there is no positive loop in $\uLeg$ (resp. $\uGcont$), together with the existence of a positive loop of contactomorphisms in $\cont_0$ \cite{allais2023spectral,allais2024spectral,discriminante,FPR,sandonmetrique,Zapolsky}. We construct new distances, see Sections \ref{se:proof} and \ref{se:autresdistances}, allowing to relax the second assumption and get the following. 

\begin{theo}\label{thm}
Suppose $\uLeg(\Lambda_*)$ is orderable and that there is a positive loop of Legendrians $(\Lambda_*^t)_{t\in S^1}\subset\Leg(\Lambda_*)$ that can be extended to a loop of contactomorphisms $(\varphi_t)_{t\in S^1}\subset\cont_0(M,\xi)$ based at the identity, i.e. $\varphi_t(\Lambda_*^0)=\Lambda_*^t$ for all $t\in S^1=\R/\Z$, then there exists an unbounded invariant distance on $\uLeg(\Lambda_*)$. 
\end{theo}
 The novelty in Theorem \ref{thm} is that we do not request the loop of contactomorphisms $(\varphi_t)_{t\in S^1}$ to be positive. We construct explicitly two, \textit{a priori}, different distances to prove Theorem \ref{thm} in Theorem \ref{thm2} and Theorem \ref{thm:distance spectrale} (see also Remark \ref{rem:pseudodistance}). The first one is a generalization of the Legendrian Fraser-Polterovich-Rosen distance \cite{allais2023spectral,FPR} and the second one is some invariantization of the spectral distance \cite{allais2023spectral,Nakamura2023}. We refer the reader directly to Section \ref{se:proof} and Section \ref{se:spectral distance} for more details.\\
 
 As a corollary we get the first example of a closed Legendrian sitting inside a contact manifold such that $\cont_0$ is orderable, i.e. it does not contain any positive loop of contactomorphisms, but $\uLeg$ admits an unbounded invariant distance. The contact manifold considered is the unitary cotangent of the torus $\T^n=\R^n/\Z^n$ endowed with its canonical contact distribution. This contact manifold can be thought as $(\T^n\times \sphere{n-1},\ker\alpha)$ where $\sphere{n-1}=\{(p_1,\cdots,p_n)\in\R^n\ |\ \sum\limits_{i=1}^np_i^2=1\}\subset \R^{n}$ is the standard Euclidean sphere and $\alpha(q,p)(Q,P):=\left<p,Q\right>=\sum\limits_{i=1}^n p_iQ_i$ for any $(q,p,Q,P)\in T(\T^n\times \sphere{n-1})$. It is shown in \cite{chernovnemirovski1} that $\conto(\T^n\times \sphere{n-1},\ker\alpha)$ is orderable. 

\begin{coro*}\label{cor : ex}
For any $n\geq 2$, the Legendrian  $\Lambda_*:=\T^{n-1}\times\{[0]\}\times\{(0,\cdots,0,1)\}$ lying inside $(\T^n\times \sphere{n-1},\ker\alpha)$ is such that $\uLeg(\Lambda_*)$ admits an unbounded invariant distance. 
\end{coro*}
\begin{proof}[Proof of the Corollary]

The Legendrian isotopy \[(\Lambda_*^t):=(\varphi_\alpha^t(\Lambda_*))=(\T^{n-1}\times\{[t]\}\times\{(0,\cdots,0,1)\})\] is a positive loop in $\Leg(\Lambda_*)$, where $\varphi_\alpha^t$ denotes the Reeb flow of $\alpha$ at time $t\in\R$ (see Section \ref{se:preliminaire} for a definition). Moreover, since $\Lambda_*$ is a hypertight Legendrian -- i.e.  the Reeb flow of $\alpha$ does not admit any contractible orbit and the non trivial Reeb trajectories starting and ending on $\Lambda_*$ represent elements in $\pi_1(M,\Lambda_*)\setminus\{0\}$ -- $\uLeg(\Lambda_*)$ is orderable thanks to \cite[Theorem 1.11]{CCR}. Finally, one can extend this loop of Legendrians to the loop of contactomorphisms $(\varphi_t)_{t\in S^1}$ defined for all $(t,q,p)\in S^1\times \T^n\times\sphere{n-1}$ by:
$\varphi_t(q_1,\cdots,q_n,p_1,\cdots,p_n)=(q_1,\cdots,q_n+[t],p_1,\cdots,p_n).$
\end{proof}

 As with the other invariant distances constructed to date on $\uLeg(\Lambda_*)$, the distance we construct in Section \ref{se:proof} takes integer values and thus induces the discrete topology. In a certain sense, we show that this has to be the case. More precisely, we say that a distance $D$ on a set $X$ is discrete if it induces the discrete topology, i.e. for all $x\in X$ there exists $C>0$ such that $D(x,y)>C$ for any $y\ne x$. We say that a distance $D$ is strongly discrete if there exists $C>0$ such that $D(x,y)>C$ for all $x\ne y$.

\begin{theo}\label{thm:discreteness}\
    \begin{enumerate}[1.]
        \item If $D$ is an invariant distance on $\Leg(\Lambda_*)$  then $D$ is strongly discrete. 
        \item If $D$ is an invariant distance on $\uLeg(\Lambda_*)$ then there exists $C>0$ such that  $D(\tLambda,\tLambda')>C$ for any $\tLambda,\tLambda'$satisfying $\Pi(\tLambda')\ne\Pi(\tLambda)$, where $\Pi :\uLeg(\Lambda_*)\to\Leg(\Lambda_*)$ is the covering map. 
     \end{enumerate}
\end{theo}

\begin{rem}
    The proof we give of Theorem \ref{thm:discreteness} in Section \ref{se:discret} also holds for distances $D$ on $\Leg$ (resp. $\uLeg$) that are $\Gcont$-invariant (resp. $\uGcont$-invariant), i.e. $D(\varphi(\Lambda_1),\varphi(\Lambda_2))=D(\Lambda_1,\Lambda_2)$ for all $\varphi\in\Gcont$, $\Lambda_1,\Lambda_2\in\Leg$ (resp. $D(\tphi\cdot\tLambda_1,\tphi\cdot\tLambda_2)=D(\tLambda_1,\tLambda_2)$ for all $\tphi\in\uGcont$, $\tLambda_1,\tLambda_2\in\uLeg$).
\end{rem}

\subsection{Discussion}
While Theorem \ref{thm} intensively uses contact rigidity, Theorem \ref{thm:discreteness} is a consequence of contact flexibility. Indeed, obstructions to the existence of positive (contractible) loop of Legendrians, which is a keypoint to define the distance of Theorem \ref{thm}, come from hard techniques such as generating functions \cite{bhupal,chernovnemirovski1,granja2014givental,San11} or different flavours of Floer Homology \cite{albers,CCR,EKP,EP00}. Notably, recent results by Hedicke demonstrate that the contact structure is essential for the existence of such obstructions: in the smooth setting, contractible smooth positive loops do exist \cite{hedicke2025positivepathsdiffeomorphismgroups}. Concerning the contact flexibility properties involved in the proof of Theorem \ref{thm:discreteness}, we refer the reader directly to Section \ref{se:discret}.\\

Finally, let us first point out some analogies between the properties of invariant distances on $\Leg$, $\uLeg$ and bi-invariant distances on $\Gcont$, $\uGcont$ and then contrast these with the behavior in the smooth setting. On the one hand, similarly to Theorem \ref{thm:discreteness}, in \cite{FPR} it is shown that a bi-invariant distance on $\Gcont$ has to be discrete due to contact flexibility phenomena. Moreover, as in Theorem \ref{thm}, it is known that some contact manifolds exhibit sufficient contact rigidity to ensure that $\uGcont$ and $\Gcont$ are \textit{unbounded}, i.e. they admit an unbounded bi-invariant distance \cite{allais2024spectral,Arlove2023,discriminante,FPR,sandonmetrique,Zapolsky}. Conversely, on the other hand, when $M$ is closed, $\mathrm{Diff}^c_0(M\times\R^n)$, i.e. the group of diffeomorphisms that are compactly isotopic to the identity, is \textit{bounded} \cite{BIP,Tsuboi}. This boundedness, much like the existence of contractible smooth positive loops discussed above \cite{hedicke2025positivepathsdiffeomorphismgroups}, reflects the greater flexibility in the smooth setting as opposed to the contact one. Hence, it seems reasonable to conjecture that, in contrast with Theorem \ref{thm}, $\mathcal{D}(N):=\{\varphi(N)\ |\ \varphi\in\mathrm{Diff}^c_0(M)\}$ (resp. its universal cover $\widetilde{\mathcal{D}}(N)$) should not admit any unbounded $\mathrm{Diff}^c_0(M)$-invariant (resp. $\widetilde{\mathrm{Diff}^c_0}(M)$-invariant) distance when $N\subset M$ is a closed and connected submanifold  of sufficiently big codimension.


\subsection{Organization of the paper} We start by fixing some notations and definitions in Section \ref{se:preliminaire}. Then we prove Theorem \ref{thm:discreteness} in Section \ref{se:discret}, and prove Theorem \ref{thm} in Section \ref{se:proof}. In Section \ref{se:autresdistances} we compare the distance we constructed in Section \ref{se:proof} with the spectral distance \cite{allais2023spectral,Nakamura2023} and Colin-Sandon oscillation's distance \cite{discriminante}. Finally, in Section \ref{se:openquestion}, we discuss how the distance constructed in Section \ref{se:proof} could contribute to defining a new unbounded invariant distance on $\uGcont(M,\xi)$ assuming that $\conto(M,\xi)$ does not contain a positive loop -- a question which remains open at present.
\subsection*{Acknowledgement} The author thanks Stefan Nemirovski and Sheila Sandon for valuable discussions and their constant support. 

\section{Preliminaries}\label{se:preliminaire}
A cooriented contact manifold $(M,\xi)$ is a manifold $M$ endowed with a distribution $\xi$ of hyperplanes such that $TM/\xi$ is an oriented line bundle and $\xi$ is maximally non-integrable, i.e. there exists a $1$- form $\alpha\in\Omega^1(M)$ whose kernel agrees with $\xi$ and $\alpha\wedge (\ud \alpha)^n\ne 0$ where $2n+1$ is the dimension of $M$. A Legendrian is a submanifold of $M$ tangent to $\xi$ of dimension $n$. In this article, all the Legendrians are considered to be closed and connected.  Once a contact form $\alpha\in\Omega^1(M)$ supporting the contact distribution $\xi$ is fixed, i.e. $\ker\alpha=\xi$, one can associate to it a vector field $R_\alpha$, called the Reeb vector field, uniquely defined by the equations $\alpha(R_\alpha)\equiv 1$ and $\ud \alpha(R_\alpha,\cdot)\equiv 0$. If $R_\alpha$ is complete, its flow is called the Reeb flow, and we denote by $\varphi_\alpha^t$ its flow at time $t\in\R$. For the rest of this Section \ref{se:preliminaire} we consider an arbitrary cooriented contact manifold $(M,\xi)$.

\subsection{Isotopies}\label{se: isotopies}
Let $I$ be an interval of $\R$ containing $0$, and $(\Lambda_t)_{t\in I}$ be a $I$-family of Legendrians lying inside $(M,\xi)$. We say that it is a Legendrian isotopy if there exists a smooth map $\varphi : I\times \Lambda_0\to M$ whose restriction to $\{t\}\times \Lambda_0$ is an embedding of image $\Lambda_t$. We call such map a parametrization of $(\Lambda_t)$. Similarly, a $I$-family of contactomorphisms $(\varphi_t)_{t\in I}\subset\cont(M,\xi)$ is a (compactly supported) contact isotopy if the map $I\times M \to M,\ (t,x)\mapsto \varphi_t(x)$ is smooth (and its restriction to $\{t\}\times M$ is compactly supported). When $I=[0,1]$ we drop it from the subscript and simply write $(\Lambda_t)$ or $(\varphi_t)$. 

\subsection{Contact Hamiltonians}
Let $\alpha$ be a contact form supporting the contact distribution $\xi$. The set $\chi(M,\xi):=\{X\text{ vector field of } M \ |\ \exists f : M\to\R,\ \mathcal{L}_X\alpha=f\alpha\}$ depends only on the contact structure and not the contact form $\alpha$; it consists of the vector fields whose flow, where defined, preserves the contact distribution. Denoting by $(X_t)_{t\in I}\subset \chi(M,\xi)$ the speed of a contact isotopy $(\varphi_t)_{t\in I}\subset\cont(M,\xi)$, i.e. $\frac{d}{dt}\varphi_t=X_t\circ\varphi_t$, the $\alpha$-contact Hamiltonian of $(\varphi_t)_{t\in I}$ is the time-dependant function $H : I\times M\to\R$, $t\mapsto\alpha(X_t(x))$. Conversely, since the map $\Phi_\alpha:\chi(M,\xi)\to C^\infty(M,\R),\ X\mapsto \alpha(X)$ is a bijection, to any smooth map $H\in C^\infty(I\times M,\R)$ one can associate an isotopy $(X_t:=\Phi_\alpha^{-1}(H(t,\cdot)))_{t\in I}$.  When $(X_t)_{t\in I}$ is complete, we say that its flow $(\varphi_t)_{t\in I} \subset\cont(M,\xi)$, i.e. the unique isotopy such that $\frac{d}{dt}\varphi_t=X_t(\varphi_t)$ for all $t\in I$ and $\varphi_0=\id$, is generated by the $\alpha$-contact Hamiltonian $H$. 

\subsection{Universal covers}\label{se:universal cover}
Let us fix a closed Legendrian $\Lambda_*\subset (M,\xi)$. We consider on $\mathcal{P}(\Lambda_*)$ the space of Legendrian $[0,1]$-isotopies starting at $\Lambda_*$
(resp. on $\mathcal{P}(M,\xi)$ the space of contact $[0,1]$-isotopies starting at the identity, resp. on $\mathcal{P}_c(M,\xi)$ the space of compactly supported contact $[0,1]$-isotopies starting at the identity)
the equivalence relation $\sim$ of being isotopic relatively to endpoints. More precisely, let $(x_t),(y_t)\in\mathcal{P}(\Lambda_*)$ (resp. $(x_t),(y_t)\in\mathcal{P}(M,\xi)$, resp. $(x_t),(y_t)\in\mathcal{P}_c(M,\xi)$), then $(x_t)\sim(y_t)$ if $x_1=y_1=:z$ and there exists a smooth map $\Gamma : [0,1]\times [0,1]\times \Lambda_*\to M$,  (resp. $\Gamma : [0,1]\times [0,1]\times M\to M$) such that $\Gamma(s,t,\cdot) : \Lambda_*\hookrightarrow M$ is an embedding of image a Legendrian $z_{s,t}$ (resp. $z_{s,t}:=\Gamma(s,t,\cdot) : M \to M$ is a contactomorphism, resp. a compactly supported contactomorphism) such that $z_{s,0}=\Lambda_*$ (resp. $z_{s,0}=\id$), $z_{s,1}=z$,
  $(z_{0,t})=(x_t)$ and $(z_{1,t})=(y_t)$. As a set, the universal cover $\uLeg(\Lambda_*)$ of $\Leg(\Lambda_*)$ (resp. $\widetilde{\cont}$, resp. $\uGcont$) can be thought as the quotient space $\mathcal{P}(\Lambda_*)/\sim$ (resp. $\mathcal{P}(M,\xi)/\sim$, resp. $\mathcal{P}_c(M,\xi)/\sim$) and the covering map $\Pi : \uLeg(\Lambda_*)\to \Leg(\Lambda_*)$ (resp. $\widetilde{\cont}\to\cont$, resp. $\uGcont \to\Gcont$) as $\Pi([(x_t)])=x_1$. We denote by $\pi_1(\Leg(\Lambda_*)):=\Pi^{-1}(\Lambda_*)\subset\uLeg(\Lambda_*)$ the fundamental group. Finally, timewise composition turns $\uGcont$ and $\widetilde{\conto}$ into groups, i.e. $[(\varphi_t)]\cdot[(\psi_t)]:=[(\varphi_t\circ\psi_t)]$, and  there is a natural action of $\widetilde{\cont}$ and $\uGcont$ on $\uLeg$ that can also be described by timewise composition, i.e. $[(\varphi_t)]\cdot [(\Lambda_t)]:=[(\varphi_t(\Lambda_t))]$. See the Appendix for a more topological description of $\uLeg(\Lambda_*)$. 

\subsection{Binary relation and orderability}\label{se:partial order}
 Fix a contact form $\alpha$ supporting $\xi$ and its coorientation, i.e. $\ker\alpha=\xi$ and $R_\alpha$ descends to a positive section of the oriented line bundle $TM/\xi\to M$. A Legendrian isotopy $(\Lambda_t)$ is non-negative (resp. positive) if $\text{there exists a parametrization }\varphi : [0,1]\times \Lambda_0\to M\text{ of }(\Lambda_t)\text{ such that  }\alpha\left(\frac{d}{dt}\varphi_t\right)\geq 0$ (resp. $\alpha\left(\frac{d}{dt}\varphi_t\right)> 0$). A natural binary relation $\cleq$ exists on $\Leg(\Lambda_*)$: $\Lambda\cleq \Lambda'$ if there exists a non-negative Legendrian isotopy $(\Lambda_t)$ connecting $\Lambda$ to $\Lambda'$, i.e. $\Lambda_0=\Lambda$ and $\Lambda_1=\Lambda'$. The definition of $\cleq$ depends neither on the choice of $\alpha$ nor the choice of the parametrization.  Similarly, one defines $\cleq$ on  $\uLeg(\Lambda_*)$ by: $\tLambda\cleq\tLambda'$ if there exists a non-negative isotopy $(\tLambda_t)\subset\uLeg(\Lambda_*)$, i.e. $(\Pi(\tLambda_t))$ is a non-negative Legendrian isotopy, such that $\tLambda_0=\tLambda$ and $\tLambda_1=\tLambda'$. In the same veins, one defines the binary relation $\cleq$ on $\conto(M,\xi)$ (resp. $\widetilde{\conto}(M,\xi)$) as follow: $\varphi\cleq \varphi'$ (resp. $\tphi\cleq\tphi'$) if there exists a non-negative $\alpha$-contact Hamiltonian $H : [0,1]\times M\to\R_{\geq 0}$ generating a contact isotopy $(\psi_t)$ such that $\psi_1=\varphi'\cdot\varphi^{-1}$ (resp. $[(\psi_t)]=\tphi'\cdot\tphi^{-1}$). These binary relations were first studied by Eliashberg and Polterovich \cite{EP00}. \\
 
For any $\Lambda\cleq\Lambda'\in\Leg$ (resp. $\tLambda\cleq\tLambda'\in\uLeg)$ and any $\varphi\cleq\varphi'\in\conto$ (resp. $\tphi\cleq\tphi'\in\widetilde{\conto})$ we have $\varphi(\Lambda)\cleq\varphi'(\Lambda')$ (resp. $\tphi\cdot\tLambda\cleq\tphi'\cdot\tLambda'$) \cite[Prop 2.9]{allais2023spectral}, in particular $\cleq$ on $\Leg$ (resp. $\uLeg$) is $\conto$-invariant (resp. $\widetilde{\conto}$-invariant). 
In addition, $\cleq$ on $\mathcal{O}$ -- $\mathcal{O}$ being one of the space $\Leg$, $\uLeg$, $\conto(M,\xi)$ or $\widetilde{\conto}(M,\xi)$ -- is transitive and reflexive.  If $\cleq$ is antisymmetric, hence a partial order, $\mathcal{O}$ is  said to be orderable. Antisymmetry of $\cleq$ on $\Leg$ (resp. $\uLeg$)  is equivalent to the absence of positive loops in $\Leg$ (resp. $\uLeg$) \cite[Prop 4.7]{chernovnemirovski2}. Similarly, when $M$ is closed, antisymmetry of $\cleq$ on $\conto=\Gcont$ (resp. $\widetilde{\conto}=\uGcont$) is equivalent to the absence of positive loops in $\Gcont$ (resp. $\uGcont$) \cite[Criterion 1.2.C]{EP00}.

\section{Proof of Theorem \ref{thm:discreteness}}\label{se:discret}

As mentionned in the introduction, the proof of Theorem \ref{thm:discreteness} exploits different contact flexibility phenomena. First, any closed Legendrian in a contact manifold has a neighborhood which is contactomorphic to the $1$-jet bundle of this Legendrian endowed with its standard contact distribution. Second, in the $1$-jet bundle, any two bounded domains intersecting sufficiently well the zero-section can be displaced and squeezed one into the other while preserving the zero-section. More precisely, let $\pi : J^1\Lambda \simeq T^*\Lambda\times\R\to \Lambda$ be the $1$-jet bundle of a closed manifold $\Lambda$ that we endow with its canonical contact distribution $\xi_{\mathrm{can}}:=\ker\alpha_\mathrm{can}$ where $\alpha_\mathrm{can}:=\ud z-\lambda$, $\lambda$ being the Liouville $1$-form of the cotangent $T^*\Lambda$, i.e. $\lambda(q,p)=p\circ\ud_{(q,p)}\pi$, and $z$ being the coordinate on $\R$. From now on $\Lambda_*$ is a closed and connected Legendrian of a cooriented contact manifold $(M,\xi)$. 

\begin{prop}\label{prop:weinstein}
There exists an open neighborhood $\mathcal{U}\subset M$ of $\Lambda_*$ such that $(\mathcal{U},\xi|_{\mathcal{U}})$ is contactomorphic to $(J^1\Lambda_*,\xi_{\mathrm{can}})$ via a contactomorphism $\Phi$ sending $\Lambda_*$ to the zero-section. 
\end{prop}

We say that an open set $D$ of a manifold $X$ is an open ball if it is diffeomorphic to $\R^n$, $n$ being the dimension of $X$. Also, by a slight abuse of notation we still denote by $\Lambda$ the image of the zero-section $\sigma_0 : \Lambda\to  J^1\Lambda$.

\begin{prop}\label{prop:contactflexibility}
  Let $B,B'$ be two subsets of $J^1\Lambda_*$ such that $D:=B\cap\Lambda_*$, $D':=B'\cap \Lambda_*$ are open balls of $\Lambda_*$. If $B$ is contained in a compact subset of $\pi^{-1}(D)$, there exists a compactly supported contact isotopy $(\eta_t)\subset\Gcont(J^1\Lambda_*,\xi_{\mathrm{can}})$ such that $\eta_0=\id$, $\eta_t(\Lambda_*)=\Lambda_*$ and $\eta_1(B)\subset B'$.
\end{prop}


We admit Proposition \ref{prop:weinstein} and Proposition \ref{prop:contactflexibility} for the moment. Before deducing Theorem \ref{thm:discreteness} from them, we fix some notations that we conserve until we prove Theorem \ref{thm:discreteness}.  Let $\Phi$ be a contactomorphism between an open neighborhood $\mathcal{U}\subset M$ of $\Lambda_*$ and $J^1\Lambda_*$ sending $\Lambda_*$ to the zero-section whose existence is guaranteed by Proposition \ref{prop:weinstein}. Let $B\subset J^1\Lambda_*$ be any non-empty open set of $J^1\Lambda_*$ with compact closure such that $B\cap \Lambda_*=\pi(B)$ is an open ball of $\Lambda_*$. Denoting by $\mathcal{B}:=\Phi^{-1}(B)\subset M$ and by $\tLambda_*$ the element in $\uLeg(\Lambda_*)$ that can be represented by the constant Legendrian isotopy.

\begin{cor}\label{cor:flexible}

  Let $D$ be an invariant distance on $\Leg(\Lambda_*)$ (resp. $\uLeg(\Lambda_*))$. Then
    \[D(\Lambda_*,\varphi(\Lambda_*))\leq 2 D(\Lambda_*,\Lambda')\quad \text{(resp. } D(\tLambda_*,\tphi\cdot\tLambda_*)\leq 2D(\tLambda_*,\tLambda'))\] for any $\varphi\in\Gcont(\mathcal{B},\xi|_{\mathcal{B}})\subset\Gcont(M,\xi)$ (resp. $\tphi\in\uGcont(\mathcal{B},\xi|_\mathcal{B})\subset\uGcont(M,\xi)$) and any $\Lambda'\in\Leg(\Lambda_*)\setminus\{\Lambda_*\}$ (resp. $\tLambda'\in\uLeg(\Lambda_*)\setminus\pi_1(\Leg(\Lambda_*))$).
\end{cor}
\begin{proof}[Proof of Corollary \ref{cor:flexible}]
We prove Corollary \ref{cor:flexible} in the case of $\uLeg(\Lambda_*)$, the case of $\Leg(\Lambda_*)$ being similar and easier. Since $\tLambda'\in \uLeg(\Lambda_*)\setminus\pi_1(\Lambda_*)$, there exists an open set $\mathcal{B}'\subset \mathcal{U}$ such that $\mathcal{B}'\cap \Lambda_*$ is a non-empty ball of $\Lambda_*$ and $\mathcal{B}'\cap\Pi(\tLambda')=\emptyset$. Putting $B':=\Phi(\mathcal{B}')$,  thanks to Proposition \ref{prop:contactflexibility},there exists $(\eta_t)\subset \Gcont(J^1\Lambda_*,\xi_{\mathrm{can}})$, a compactly supported contact isotopy starting at the identity,  such that $\eta_t(\Lambda_*)=\Lambda_*$ and $\eta_1(B)\subset B'$. Thus, the contact isotopy $(\psi_t)\subset\Gcont(M,\xi)$
\[\psi_t(x):=\begin{cases}
    \Phi^{-1}(\eta_t(\Phi(x))) & \text{if } x\in \mathcal{U}\\
    x &\text{otherwise},
\end{cases}\]
is such that $\psi_t(\Lambda_*)=\Lambda_*$ and $\psi_1(\mathcal{B})\subset \mathcal{B}'$. Consider now a contact isotopy $(\varphi_t)\subset\Gcont(\mathcal{B},\xi|_{\mathcal{B}})$ representing $\tphi\in\uGcont(\mathcal{B},\xi|_{\mathcal{B}})$. Then $\widetilde{\mu}:=[(\psi_1\circ\varphi_t^{-1}\circ\psi_1^{-1})]$ is an element of $\uGcont(\mathcal{B}',\xi|_{\mathcal{B}'})$ and so $\tLambda'=\widetilde{\mu}\cdot\tLambda'$. Denoting by $\tLambda=\tphi\cdot\tLambda_*$ and $\tpsi:=[(\psi_t)]$, we deduce that,
\[\begin{aligned}D(\tLambda_*,\tLambda)=D(\tpsi\cdot\tLambda_*,\tpsi\cdot\tLambda)=D(\tLambda_*,\tpsi\cdot\tLambda)&\leq D(\tLambda_*,\tLambda')+D(\tLambda',\tpsi\cdot\tLambda)\\
&=D(\tLambda_*,\tLambda')+D(\widetilde{\mu}\cdot\tLambda',\widetilde{\mu}\cdot\tpsi\cdot\tLambda)\\
&=2D(\tLambda_*,\tLambda')\qedhere
\end{aligned}\]\end{proof}

Theorem \ref{thm:discreteness} is a direct consequence of the previous corollary.

\begin{proof}[Proof of Theorem \ref{thm:discreteness}]\
 Note that by invariance of $D$ it is enough to show that there exists $C>0$ such that $D(\Lambda_*,\Lambda)>C$ for any $\Lambda\in\Leg(\Lambda_*)\setminus\{\Lambda_*\}$ (resp. $D(\tLambda_*,\tLambda)>C$ for any $\tLambda\in\uLeg(\Lambda_*)\setminus\pi_1(\Leg(\Lambda_*))$). Suppose by contradiction that there exists a sequence $(\Lambda_n)_{n\in\N}\subset\Leg(\Lambda_*)\setminus\{\Lambda_*\}$ (resp. $(\tLambda_n)_{n\in\N}\subset\uLeg(\Lambda_*)\setminus\pi_1(\Leg(\Lambda_*))$) such that the positive sequence $(D(\Lambda_*,\Lambda_n))_{n\in\N}$ (resp. $(D(\tLambda_*,\tLambda_n))_{n\in\N}$) goes to $0$ as $n$ goes to infinity. Consider $\varphi\in\Gcont(\mathcal{B},\xi|_{\mathcal{B}})$ (resp. $\tphi\in\uGcont(\mathcal{B},\xi|_{\mathcal{B}})$) such that $\varphi(\Lambda_*)\ne\Lambda_*$ (resp. $\tphi\cdot\tLambda_*\ne\tLambda_*$). By Corollary \ref{cor:flexible}, $D(\Lambda_*,\varphi(\Lambda_*))\leq 2 D(\Lambda_*,\Lambda_n)$ (resp. $D(\tLambda_*,\tphi\cdot\Lambda_*)\leq 2D(\tLambda_*,\tLambda_n)$) for all $n\in\N$, and so leads to the contradiction $D(\Lambda_*,\varphi(\Lambda_*))=D(\tLambda_*,\tphi\cdot\tLambda_*)=0$.\end{proof}

\subsection{Proof of Proposition \ref{prop:weinstein} and  \ref{prop:contactflexibility} }\label{se:preuve des prop}

To prove Proposition \ref{prop:weinstein} and Proposition \ref{prop:contactflexibility}  we use the Liouville vector field. More precisely, denote by $L : T^*\Lambda_* \to T(T^*\Lambda_*)$ the Liouville vector field of $(T^*\Lambda_*,\ud \lambda)$, i.e. the unique vector field of $T^*\Lambda_*$ satisfying $\ud\lambda(L,\cdot)=\lambda$. We call the vector field $X$ on $J^1\Lambda_*$ defined by $X((q,p),z)=(L(q,p),z)\in T_{(q,p)}T^*M\times T_z\R\simeq T_{(q,p)}T^*M\times\R$ the Liouville vector field of $J^1\Lambda_*$. Note that $X$ is a complete vector whose flow $\varphi_X^t$ at time $t\in\R$ is a contactomorphism of $(J^1\Lambda,\xi_{\mathrm{can}})$.
Moreover, for all $t\in\R$, $\varphi_X^t$ preserves the zero-section and the fibers, i.e. $\varphi_X^t(\Lambda_*)=\Lambda_*$ and $\varphi_X^t(\pi^{-1}(x))=\pi^{-1}(x)$ for all $x\in\Lambda_*$, and it can be used to squeeze the latter to the former, more formally:

\begin{lem}\label{lem:Liouville}
Let $D'\subset\Lambda_*$ be a non-empty subset of $\Lambda_*$. For any set  $B$ contained in a compact subset of $\pi^{-1}(D')$ and any open neighborhood $B'\subset J^1\Lambda_*$ of $D'$, there exists $T>0$ such that $\varphi_X^{-t}(B)\subset B'$ for all $t\geq T$. 
\end{lem}

We let the proof of Lemma \ref{lem:Liouville} to the reader and prove Proposition \ref{prop:contactflexibility}.

\begin{proof}[Proof of Proposition \ref{prop:contactflexibility}]
Since the sets $D$ and $D'$ are two open balls of $\Lambda_*$ there exists a smooth isotopy $(\varphi_t)\subset \mathrm{Diff}(\Lambda_*)$ such that $\varphi_0=\id$ and $\varphi_1(D)\subset D'$. This isotopy lifts to the contact isotopy $(\psi_t)\subset\cont_0(J^1\Lambda_*,\xi_{\mathrm{can}})$ defined by $\psi_t(q,p,z)=(\varphi_t(q),p\circ(\ud_q \varphi_t)^{-1},z)$. Therefore, $\psi_1(B)\subset \pi^{-1}(D')$. Thus, thanks to Lemma \ref{lem:Liouville}, there exists $T>0$ big enough so that $\varphi_X^{-T}(\psi_1(B))\subset B'$. Consequently, the contact isotopy $(\varphi_X^{-tT}\circ\psi_t)$ preserves $\Lambda_*$ and its time-one sends $B$ inside $B'$. By multiplying the $\alpha_\mathrm{can}$-contact Hamiltonian of $(\varphi_X^{-tT}\circ\psi_t)$ by a compactly supported function $\rho : J^1\Lambda_*\to \R$ which is equal to $1$ on an open neighborhood of $\Lambda_*\bigcup (\underset{t\in[0,1]}\bigcup\varphi_X^{-tT}\circ\psi_t(B))$, we get a new $\alpha_{\mathrm{can}}$-contact Hamiltonian generating a contact isotopy $(\eta_t)$ which has the desired properties\end{proof}

Proposition \ref{prop:weinstein} is well known in the folklore. As we are not aware of any written proof in the literature, we provide one below.\\ 

 For any Riemaniann metric $g$ on $\Lambda_*$, we denote by $D_gT^*\Lambda_*:=\{(x,u)\in T^*\Lambda_*\ |\ g_*(u,u)<1\}$ the co-disc bundle, where $g_*$ is the dual metric on $T^*\Lambda_*$, and,  for any $\varepsilon>0$, we denote by $U_g^\varepsilon:=D_gT^*\Lambda_*\times(-\varepsilon,\varepsilon)\subset J^1\Lambda_*$.
 
\begin{lem}\label{prop:star-shaped}
   For any $\varepsilon>0$ and any Riemaniann metric $g$ on $\Lambda_*$ there exists a contactomorphism $\Phi : (U_g^\varepsilon,\xi_{\mathrm{can}}|_{U_g^\varepsilon}) \to (J^1\Lambda_*,\xi_{\mathrm{can}})$ that restricts to the identity on the zero-section.  
\end{lem}

 For the sake of completeness we prove Lemma  \ref{prop:star-shaped} by repeating almost verbatim the proof of Proposition 3.1 in \cite{Chekanov_vanKoert_Schlenk2008}. 

\begin{proof}   Putting $U_t:=\varphi_X^t(U_g^\varepsilon)$, then $U_{t_1}\subset U_{t_2}$, whenever $t_1\leq t_2$, and $\underset{t\in \R}\bigcup U_t=J^1\Lambda_*$. 

\textbf{First step:} Let us prove that given any $a<b$ and any $c<d$, there exists $\Phi_{a,b}^{c,d}\in\cont_0(J^1\Lambda_*,\xi_{\mathrm{can}})$ that coincides with $\varphi_{X}^{c-a}$ on a neighborhood of $\partial U_a$ and with $\varphi_{X}^{d-b}$ on a neighborhood of $\partial U_b$. Since $J^1\Lambda_*\setminus \partial U_t$ has exactly two connected components for any $t\in\R$, such a contactomorphism would send $ U_a$ to $U_c$ and send $U_b$ to $U_d$. To construct $\Phi_{a,b}^{c,d}$, let $r_1,r_2$ be such that $a<r_1<r_2<\min(b,d-(c-a))$. Let $F :J^1\Lambda_*\to\R$ be a $\alpha_{\mathrm{can}}$-Hamiltonian such $F=0$ on $U_{r_1}$ and $F=H_X$ outside $U_{r_2}$, where $H_X(q,p,z)=z$ is the $\alpha_{\mathrm{can}}$-contact Hamiltonian of $X$. The contact flow $(\varphi_F^t)_{t\in\R}$ generated by $F$ is such that $\varphi_F^{(d-b)-(c-a)}$ coincides with the identity in a neighborhood of $\partial U_a$ and coincides with $\varphi_X^{(d-b)-(c-a)}$ on a neighborhood of $\partial U_b$, therefore one can pose $\Phi_{a,b}^{c,d}:=\varphi_X^{c-a}\circ\varphi_F^{(d-b)-(c-a)}$.

\textbf{Second step:} Since $U_g^\varepsilon=\underset{n\in\N_>0}\bigcup U_{\frac{-1}{n}}$ and $J^1\Lambda=\underset{n\in\N_{>0}}\bigcup U_n$ we construct the desired contactomorphism $\Phi$ as follow
\[\Phi : U\to J^1\Lambda\quad \quad \quad x\mapsto \begin{cases}
    \varphi_X^{2}(x) & \text{if } x\in U_{-1}\\
    \Phi_{\frac{-1}{n},\frac{-1}{n+1}}^{n,n+1}(x) & \text{if } x\in U_{\frac{-1}{n+1}}\setminus U_{\frac{-1}{n}}.  \qedhere
\end{cases}  \]
\end{proof}

\begin{proof}[Proof of Proposition \ref{prop:weinstein}] By the standard contact Weinstein neighborhood Theorem (Corollary 4.1 \cite{Lychagin1977}) there exists a neighborhood $\mathcal{V}\subset M$ of the Legendrian $\Lambda_*\subset (M,\xi)$, a neighborhood $V\subset (J^1\Lambda_*,\xi_{\mathrm{can}})$ of the zero-section and $\Psi : (\mathcal{V},\xi|_{\mathcal{V}})\to (V,\xi_{\mathrm{can}})$ a contactomorphism sending $\Lambda_*$ to the zero-section. There exists a Riemannian metric $g$ on $\Lambda_*$ and $\varepsilon>0$ such that $U_g^\varepsilon$ is contained in $V$. So by Lemma \ref{prop:star-shaped}, the neighborhood $\mathcal{U}:=\Psi^{-1}(U_g^\varepsilon)\subset (M,\xi)$ of $\Lambda_*$ is contactomorphic to $(J^1\Lambda_*,\xi_{\mathrm{can}})$ by a contactomorphism, i.e. $\Phi\circ\Psi$, sending $\Lambda_*$ to the zero-section. 
\end{proof}

\section{Proof of Theorem \ref{thm}}\label{se:proof}
In Section \ref{se:proof} $\Lambda_*$ is a closed and connected Legendrian submanifold of a cooriented contact manifold $(M,\xi=\ker\alpha)$ such that $\uLeg(\Lambda_*)$, $\uLeg$ for simplicity, is orderable and such that there exists a loop of Legendrians $(\Lambda_*^t)_{t\in S^1}$, based at $\Lambda_*$, which is positive and that can be extended to a loop of contactomorphisms $(\varphi_t)_{t\in S^1}\subset\cont(M,\xi)$ based at the identity. We denote by $\tLambda_*\in\uLeg$ the constant isotopy equivalence class and for any $T\in\R$  we denote by $\tphi_T:=[(\varphi_{tT})_{t\in[0,1]}]\in\widetilde{\conto}(M,\xi)$ and $\tLambda_*^T:=\tphi_T\cdot\tLambda_*$. We define the following maps $\ell_\pm : \uLeg\times\uLeg \to \Z\cup\{+\infty\}\cup\{-\infty\}$
\[\ell_+(\tLambda_1,\tLambda_0):=\inf\{N\in\Z\ |\ \tLambda_1\cleq\tphi_N\cdot\tLambda_0\}\text{ and } \ell_-(\tLambda_1,\tLambda_0):=\sup\{N\in\Z\ |\ \tphi_N\cdot\tLambda_0\cleq \tLambda_1\}\]
where by definition the $\inf\emptyset=+\infty$ and $\sup\emptyset=-\infty$. Theorem \ref{thm} is a direct consequence of the following. 
\begin{thm}\label{thm2}
    The maps $\ell_\pm$ take value in $\Z$ and the map $\ud : \uLeg\times\uLeg\to\N_{\geq 0}$ 
    \[(\tLambda_0,\tLambda_1)\mapsto \ud(\tLambda_0,\tLambda_1):=
       \max\{\ell_+(\tLambda_1,\tLambda_0),-\ell_-(\tLambda_1,\tLambda_0)\} \] is an invariant unbounded distance which is compatible with the order, i.e. $x\cleq y\cleq z\Rightarrow \ud(x,y)\leq \ud(x,z)$. 
\end{thm}

The next Propositions of this section aim to prove Theorem \ref{thm2}.

\begin{prop}\label{prop1}
$\ell_\pm(\tLambda_1,\tLambda_0)\in\Z$ and $\ell_-(\tLambda_1,\tLambda_0)\leq \ell_+(\tLambda_1,\tLambda_0)$ for any $\tLambda_1,\tLambda_0\in\uLeg$.
\end{prop}

\begin{proof}
First, recall from \cite[Proposition 3.1]{allais2023spectral} that for any $\tLambda\in\uLeg$ 
\begin{equation*}\label{eq:crucial}\ell_+^{\underline{\tLambda_*}}(\tLambda):=\inf\{T\in\R\ |\ \tLambda\cleq \tLambda_*^T\}>-\infty\ \text{ and }\ \ell_-^{\underline{\tLambda_*}}(\tLambda):\sup\{T\in\R\ |\ \tLambda_*^T\cleq\tLambda\}<+\infty.
   \end{equation*}


Now let $\tpsi\in\widetilde{\conto}$ be such that $\tpsi\cdot\tLambda_*=\tLambda_0$ and $T_0:=\max\{\ell_+^{\underline{\tLambda_*}}(\tpsi^{-1}\tLambda_1),-\ell_-^{\underline{\tLambda_*}}(\tpsi^{-1}\tLambda_1)\}$. Note that $T_0\in\R_{\geq 0}$ is non-negative since $\ell_-^{\underline{\tLambda_*}}\leq \ell_+^{\underline{\tLambda_*}}$ by transitivity and antisymmetry of $\cleq$. Denote by $N:=\lfloor T_0\rfloor+1$. Thus
\begin{equation}\label{eq}\tphi_{-N}\cdot\tLambda_*=\tLambda_*^{-N}\cleq \tpsi^{-1}\tLambda_1\cleq \tLambda_*^N=\tphi_{N}\cdot\tLambda_*.
   \end{equation}
  Therefore, using the invariance of $\cleq$ and the fact that $\tphi_N$ lies in the center of $\widetilde{\conto}$:
   \begin{equation}\label{eq2}
   \begin{aligned} \tphi_{-N}\cdot\tLambda_*\cleq \tpsi^{-1}\cdot\tLambda_1\cleq \tphi_N\cdot\tLambda_* &\Rightarrow \tpsi\cdot \tphi_{-N}\cdot \tLambda_*\cleq\tLambda_1\cleq \tpsi\cdot\tphi_N\cdot\tLambda_*\\
   &\Rightarrow\tphi_{-N}\cdot\tpsi\cdot\Lambda_*\cleq\tLambda_1\cleq\tphi_N\cdot\tpsi\cdot\tLambda_*\\
   &\Rightarrow \tphi_{-N}\cdot\tLambda_0\cleq\tLambda_1\cleq \tphi_{N}\cdot\tLambda_0.
   \end{aligned}\end{equation}
 From \eqref{eq2} and the antisymmetry of $\cleq$ we deduce that
   $\ell_+(\tLambda_1,\tLambda_0)\geq-N$ and $\ell_-(\tLambda_1,\tLambda_0)\leq N$ which proves the finiteness of $\ell_\pm(\tLambda_1,\tLambda_0)$. Finally, we claim that $
\tphi_{N_1}\cdot\tLambda_0\cleq \tphi_{N_2}\cdot\tLambda_0\quad \text{for any }N_1\leq N_2\in\Z.$ From this claim, we deduce directly the desired inequality $\ell_-(\tLambda_1,\tLambda_0)\leq\ell_+(\tLambda_1,\tLambda_0)$ using the transitivity and antisymmetry of $\cleq$. To prove the claim, note that $\tphi_{N_1}\cdot\tLambda_*=\tLambda_*^{N_1}\cleq \tLambda_*^{N_2}=\tphi_{N_2}\cdot\tLambda_*$ for any $N_1\leq N_2\in \Z$. Thus, $\tpsi\cdot\tphi_{N_1}\cdot\tLambda_*\cleq \tpsi\cdot\tphi_{N_1}\cdot\tLambda_*$ by invariance of $\cleq$. Again, $\tphi_{N}$ being central in $\widetilde{\conto}$, when $N\in\Z$, allows to conclude the proof.\end{proof}

\begin{prop}\label{prop2}
    
    $\ell_+(\tLambda_1,\tLambda_0)=-\ell_-(\tLambda_0,\tLambda_1)$ for any $\tLambda_0,\tLambda_1\in\uLeg$.
\end{prop}
\begin{proof}
\begin{equation*}
\begin{aligned}\ell_+(\tLambda_1,\tLambda_0)=\inf\{N\in \Z\ |\ \tLambda_0\cleq\tphi_N\cdot\tLambda_1\}&=\inf\{N\in\Z\ |\ \tphi_{-N}\cdot\tLambda_0\cleq\tLambda_1\}\\
    &=\inf\{-N\in\Z\ |\ \tphi_N\cdot\tLambda_0\cleq\tLambda_1\}\\
    &=-\sup\{N\in\Z\ |\ \tphi_N\cdot\tLambda_0\cleq\tLambda_1\}. \qedhere
    \end{aligned}
\end{equation*}
\end{proof}
\begin{prop}\label{prop3}
    For any $\tLambda_0,\tLambda_1,\tLambda_2\in\uLeg$ 
    \[\ell_+(\tLambda_0,\tLambda_2)\leq\ell_+(\tLambda_0,\tLambda_1)+\ell_+(\tLambda_1,\tLambda_2)\text{ and } \ell_-(\tLambda_0,\tLambda_2)\geq \ell_-(\tLambda_0,\tLambda_1)+\ell_-(\tLambda_1,\tLambda_2). \]
\end{prop}
\begin{proof}
    Let $N_0=\ell_+(\tLambda_0,\tLambda_1)$ and $N_1=\ell_+(\tLambda_1,\tLambda_2)$ then 
    $\tLambda_0\cleq\tphi_{N_0}\cdot\tLambda_1\text{ and } \tLambda_1\cleq \tphi_{N_1}\cdot \tLambda_2$ which implies using transitivity and invariance of $\cleq$ that $\tLambda_0\cleq \tphi_{N_0+N_1}\cdot\Lambda_2$, hence the triangle inequality satisfied by $\ell_+$. The same reasoning yields the reverse triangle inequality satisfied by $\ell_-$. 
\end{proof}

\begin{prop}\label{prop4}
$\ell_\pm(\teta\cdot\tLambda_1,\teta\cdot\tLambda_0)=\ell_\pm(\tLambda_1,\tLambda_0)$ for all $\teta\in\widetilde{\conto}$ and all $\tLambda_1,\tLambda_0\in\uLeg$. 
\end{prop}

\begin{proof}
    We do it only for $\ell_+$ since the same arguments hold for $\ell_-$.
    \[\begin{aligned}\ell_+(\teta\cdot\tLambda_1,\teta\cdot\tLambda_0)&=\inf\{N\in\Z\ |\ \teta\cdot \tLambda_1\cleq\tphi_N\cdot \teta\cdot \tLambda_0\}\\
    &=\inf\{N\in\Z\ |\ \teta\cdot\tLambda_1\cleq\teta\cdot\tphi_N\cdot\tLambda_0\}\\
    &=\inf\{N\in\Z\ |\ \tLambda_1\cleq\tphi_N\cdot\tLambda_0\},
    \end{aligned}\]
    where the second equality holds because $\tphi_N$ is in the center of $\widetilde{\conto}$ and the third one because $\cleq$ is invariant. 
\end{proof}

\begin{proof}[Proof of Theorem \ref{thm2}]\

\begin{enumerate}
    \item Since $\ell_-\leq \ell_+$, thanks to Proposition \ref{prop1} the map $\ud$ takes value in $\N_{\geq 0}$. Moreover, for any $\tLambda_1,\tLambda_0\in\uLeg$\[\ud(\tLambda_0,\tLambda_1)=0\quad \Leftrightarrow \quad \tLambda_0\cleq \tLambda_1\cleq\tLambda_0 \quad \Leftrightarrow \quad \tLambda_1=\tLambda_0.\]
    \item Thanks to Proposition \ref{prop2} the map $\ud$ is symmetric, i.e. $\ud(\tLambda_1,\tLambda_0)=\ud(\tLambda_0,\tLambda_1)$.
    \item Thanks to Proposition \ref{prop3} the map $\ud$ satisfies the triangle inequality.
    \item Thanks to Proposition \ref{prop4} the map $\ud$ is $\widetilde{\conto}$ invariant.
    \end{enumerate}
Hence, $\ud$ is an invariant distance. Its compatibility with the partial order follows directly from its definition and the transitivity of $\cleq$. To complete the proof it remains to prove that $\ud$ is unbounded. For all $T>0$
\[\begin{aligned}\ell_+(\tLambda^T_*,\tLambda_*)=\inf\{N\in\Z\ |\ \tLambda^T_*\cleq\tphi_N\cdot\tLambda_*\}=\inf\{N\in\Z\ |\ \tLambda^T_*\cleq\tLambda_*^N\}=\lceil T\rceil
\end{aligned},\]
thus
$\ud(\tLambda_*^T,\tLambda_*)\geq T$ and $\ud$ is indeed unbounded.
\end{proof}

\section{Applications to the study of other invariant distances on $\uLeg$ and $\uGcont$}\label{se:autresdistances}
In this Section \ref{se:autresdistances} as before $\Lambda_*$ denotes a closed and connected Legendrian of a cooriented contact manifold $(M,\xi)$.

\subsection{Comparison with the spectral distance}\label{se:spectral distance}
 For simplicity, in this Subsection \ref{se:spectral distance}, we assume that $M$ is closed even if similar considerations could be carried on for open $M$. For any contact form $\alpha$ supporting $\xi$, and any $T\in\R$, we denote by $\varphi_\alpha^T\in\Gcont$ its Reeb flow at time $T$ and by $\tphi_\alpha^T\in\uGcont$ the element that can be represented by the path $(\varphi_\alpha^{tT})\subset\Gcont$. Recall from \cite{allais2023spectral} that when $\uLeg(\Lambda_*)$, or simply $\uLeg$, is orderable then the maps $\ell_\pm^\alpha : \uLeg\times\uLeg\to\R\cup\{\mp\infty\}$
\[\ell_+^\alpha(\tLambda_1,\tLambda_0):=\inf\{t\in\R \ |\ \tLambda_1\cleq\tphi_\alpha^t\cdot\tLambda_0\} \text{ and } \ell_-^\alpha(\tLambda_1,\tLambda_0):=\sup\{t\in\R\ |\ \tphi_\alpha^t\cdot\tLambda_0\cleq\tLambda_1\}\]
take value in $\R$. These maps are then used to construct an unbounded pseudo-distance $\dspec^\alpha :\uLeg\times\uLeg\to\R_{\geq 0}, (\tLambda_1,\tLambda_0)\mapsto \max\{\ell_+^\alpha(\tLambda_1,\tLambda_0),-\ell_-^\alpha(\tLambda_1,\tLambda_0)\}$ which is not invariant. A natural way to construct an invariant map from $\dspec^\alpha$ is the following
\begin{equation*}
    \ud_\infty^\alpha : \uLeg\times\uLeg\to\R_{\geq 0}\cup\{+\infty\} \quad \quad \quad (\tLambda_1,\tLambda_0)\mapsto \sup\{\dspec^\alpha(\tpsi\cdot\tLambda_1,\tpsi\cdot\tLambda_0)\ |\ \tpsi\in\uGcont\}
\end{equation*}

From now on until the end of Subsection \ref{se:spectral distance} we suppose that the hypothesis of Theorem \ref{thm} holds, i.e. $\uLeg$ is orderable and there is a positive loop of Legendrians $(\Lambda_*^t)_{t\in S^1}\subset\Leg$ that can be extended to a loop of contactomorphisms $(\varphi_t)_{t\in S^1}\subset\Gcont$ based at the identity, and we denote by $\ud$ the invariant distance constructed in Section \ref{se:proof} using the loop $(\varphi_t)_{t\in S^1}$.
\begin{thm}\label{thm:distance spectrale}
The map $\ud_\infty^\alpha$ takes value in $\R_{\geq 0}$ and thus is an unbounded invariant pseudo-distance. Moreover $(\uLeg,\ud)$ and $(\uLeg,\ud_\infty^\alpha)$ are quasi-isometric, i.e. there exist constants $A\geq 0$ and $B\geq 1$ such that $\frac{1}{B}\ud-A \leq \ud_\infty^\alpha\leq B\ud+A$. 
\end{thm}

\begin{rem}\label{rem:pseudodistance} In the previous Theorem \ref{thm:distance spectrale} we abusively referred to a quasi-isometry although $\ud_\infty^\alpha$ was merely a pseudo-distance. However the above statement can be made precise since any invariant pseudo-distance $D:\uLeg\times\uLeg\to\R_{\geq 0}$ can be turned into a genuine invariant distance $\overline{D}$, which is stricly discrete, by setting $\overline{D}(\tLambda_0,\tLambda_1):=\max \{D(\tLambda_0,\tLambda_1), 1\}$ if $\tLambda_0\ne\tLambda_1$, and $0$ otherwise.
\end{rem}

$\uLeg$ being orderable implies that $\uGcont$ is orderable and therefore the maps $c_\pm^\alpha : \uGcont \to\R\cup\{\mp\infty\}$ defined by $c_+^\alpha(\tphi):=\inf\{t\in\R\ |\ \tphi\cleq\tphi_\alpha^t\} \text{ and } c_-^\alpha(\tphi):=\sup\{t\in\R\ |\ \tphi_\alpha^t\cleq\tphi\}$
take value in $\R$ and the map $\Nspec{\cdot}^\alpha : \uGcont\to\R_{\geq 0},\ \tphi\mapsto \max\{c_+^\alpha(\tphi),-c_-^\alpha(\tphi)\}$ is an unbounded pseudo-norm on $\uGcont$ (see \cite{allais2023spectral,ArlovePhD}). Before proving Theorem \ref{thm:distance spectrale} let us prove the following. 
\begin{prop}\label{prop : distance spectrale'}
For any $\tLambda_1,\tLambda_0\in\uLeg$ we have $\dspec^\alpha(\tLambda_1,\tLambda_0)\leq \Nspec{\tphi_1}^\alpha\ud(\tLambda_1,\tLambda_0)$,
    where $\tphi_1\in\uGcont$ denotes the lift of the loop $(\varphi_t)_{t\in S^1}$.
\end{prop}

\begin{proof}
    Let $c:=\Nspec{\tphi_1}^\alpha$ and $N:=\ud(\tLambda_1,\tLambda_0)$. By triangle inequality $\Nspec{\tphi_N}^\alpha\leq Nc$. This implies that $\tphi_N\cleq\tphi_\alpha^{Nc+\varepsilon}$ and $\tphi_\alpha^{-(Nc+\varepsilon)}\cleq \tphi_{-N}$ for any $\varepsilon>0$. Therefore \[\tLambda_1\cleq\tphi_N\cdot \tLambda_0\cleq \tphi_\alpha^{Nc+\varepsilon}\cdot\tLambda_0 \quad \text{ and } \quad \tphi_\alpha^{-(Nc+\varepsilon)}\cdot\tLambda_0\cleq \tphi_{-N}\cdot\tLambda_0\cleq \tLambda_1.\]
So $\dspec^\alpha(\tLambda_1,\tLambda_0)\leq Nc+\varepsilon$ for any $\varepsilon>0$ and the result follows. \end{proof}


\begin{proof}[Proof of Theorem \ref{thm:distance spectrale}]
The fact that $\ud_\infty^\alpha$ takes value in $\R_{\geq 0}$ is a direct consequence of Proposition \ref{prop : distance spectrale'} and the invariance of $\ud$. Let us prove that $(\uLeg,\ud)$ and $(\uLeg,\ud_\infty^\alpha)$ are quasi-isometric. By invariance of $\ud$ and $\ud_\infty^\alpha$, it is enough to prove the existence of $A\geq 0$ and $B\geq 1$ satisfying
$
    \frac{1}{B}\ud(\tLambda,\tLambda_*)-A\leq\ud_\infty^\alpha(\tLambda,\tLambda_*)\leq B\ud (\tLambda,\tLambda_*)+A$ for all pairs $(\tLambda,\tLambda_*)\in\uLeg^2$ so that $\tLambda\in\uLeg$ is arbitrary and $\tLambda_*\in\uLeg$ is the constant isotopy equivalence class. Let  $H : S^1\times M\to\R$ be the $\alpha$-Hamiltonian function of $(\varphi_t)_{t\in S^1}$. Let $m:=\min\{H_s(x)\ |\ (s,x)\in\underset{t\in S^1}\bigcup\{t\}\times\Lambda_*^t\}>0$. We claim that $\ud(\tLambda,\tLambda_*)\leq C\left\lceil\dspec^\alpha(\tLambda,\tLambda_*)\right\rceil+1$, where $C:=\lceil\frac{1}{m}\rceil$. Admitting this claim it follows that $\ud(\tLambda,\tLambda_*)\leq C(\ud_\infty^\alpha(\tLambda,\tLambda_*)+1+\frac{1}{C})$. Combining the latter inequality with $ \ud_\infty^\alpha(\tLambda,\tLambda_*)\leq \Nspec{\tphi_1}^\alpha\ud(\tLambda,\tLambda_*)$ coming from Proposition \ref{prop : distance spectrale'} we deduce that the constants $A:=1+\frac{1}{C}$ and $B:=\max\{1,C,\Nspec{\tphi_1}^\alpha\}$ are the desired ones. It remains to prove the claim. Denote by $\ell_\pm^\alpha:=\ell_\pm^\alpha(\tLambda,\tLambda_*)$. Note that since $\dspec^\alpha(\tLambda,\tLambda_*)=\max\{\ell_+^\alpha,-\ell_-^\alpha\}$ and $\ell_-^\alpha\leq\ell_+^\alpha$, we can treat only the cases where $\ell_+^\alpha\geq 0$ and $\ell_-^\alpha\leq 0$. Since,
$\tphi_\alpha^{Tm}\cdot \tLambda_*\cleq\tphi_T\cdot\tLambda_*$ if $T\geq 0$, and $\tphi_T\cdot \tLambda_*\cleq \tphi_\alpha^{Tm}\cdot \tLambda_*$ if $T\leq 0$, we get  $\ell_+(\tLambda,\tLambda_*)\leq \left\lceil \ell_+^\alpha/m\right\rceil+1 $, and $\ell_-(\tLambda,\tLambda_*)\geq \left\lfloor \ell_-^\alpha/m\right\rfloor-1$. Hence,
$\ud(\tLambda,\tLambda_*)\leq\max\left\{\left\lceil \ell_+^\alpha/m\right\rceil,-\left\lfloor \ell_-^\alpha/m\right\rfloor\right\}+1$, and the claim follows. \end{proof}


\subsection{Comparison with the oscillation distance of Colin-Sandon}\label{sec : discriminante}
Let us recall the definition of $\dCSosc$. A Legendrian isotopy $(\Lambda_t)$ is monotone, if it is either a positive Legendrian isotopy or a negative one, and it is embedded if $\Lambda_{t_1}\cap\Lambda_{t_2}=\emptyset$ for all $t_1\ne t_2$. If $(\Lambda_t)$ and $(\Lambda_t')$ are two Legendrian isotopies such that $\Lambda_1=\Lambda_0'$, we define the concatenation of these two isotopies $(\Lambda_t*\Lambda_t')$ to be
\begin{equation}\label{eq:concatenation}
    t\mapsto
    \begin{cases}
        \Lambda_{a(2t)} & t\in[0,1/2],\\
        \Lambda_{a(2t-1)}' & t\in [1/2,1],
    \end{cases}
\end{equation}
where $a:[0,1]\to [0,1]$ is a surjective non decreasing smooth function which is equal to $0$ in a neighborhood of $0$ and $1$ in a neighborhood of $1$. Moreover, for arbitrary $\tLambda\in\uLeg(\Lambda_*)$,  $\tLambda'\in\Leg(\Pi(\tLambda))$, by concatenating any two of their representatives, one can define without ambiguity $\tLambda*\tLambda'$. For any Legendrian $\Lambda_0\in\Leg(\Lambda_*)$, and any $\tLambda\in\uLeg(\Lambda_0)$, Colin-Sandon in \cite{discriminante} defined the number $\losc^+(\tLambda)$ by
\begin{equation*}\label{eq:loscp}
    \losc^+(\tLambda) := \min\left\{ k\in\N\ \left|\
        \parbox{7.3cm}{there exist embedded monotone isotopies $(\Lambda_t^1),\ldots,(\Lambda_t^n)$,
            $k$ of which are positive,
    such that $[(\Lambda_t^1*\cdots*\Lambda_t^n)]=\tLambda\in\uLeg(\Lambda_0)$}\right.\right\}\in\N\bigcup\{+\infty\},
\end{equation*}
with conventions $\min\emptyset=+\infty$ and $\losc^+(\tLambda)=0$ if $\tLambda$ is can be represented by the constant isotopy. Finally, we denote by $\tLambda^{-1}\in\uLeg(\Lambda_1)$ the equivalence class of $(\Lambda_{1-t})$, $(\Lambda_t)$ being any representative of $\tLambda$.
\begin{prop}[\cite{discriminante}]
  If $\uLeg(\Lambda_*)$ is orderable then the map $\dCSosc : \uLeg(\Lambda_*)\times\uLeg(\Lambda_*)\to \N\cup\{+\infty\}$ defined by $\dCSosc(\tLambda,\tLambda'):=\losc^+(\widetilde{\gamma})+\losc^+(\widetilde{\gamma}^{-1})$ for any $\tLambda,\tLambda'\in\uLeg(\Lambda_*)$ and any $\widetilde{\gamma}\in\uLeg(\Pi(\tLambda))$ such that $\tLambda*\widetilde{\gamma}=\tLambda'$, takes value in $\N$ and is a distance compatible with the partial order, i.e. $x\cleq y \cleq z\Rightarrow\dCSosc(x,y)\leq \dCSosc(x,z)$. 
\end{prop}

\begin{prop}\label{prop : distance spectrale}
    Suppose $\uLeg(\Lambda_*)$ is orderable and that there is a positive loop of Legendrians $(\Lambda_*^t)_{t\in S^1}\subset\Leg(\Lambda_*)$ that can be extended to a loop of contactomorphisms $(\varphi_t)_{t\in S^1}\subset\Gcont$ based at the identity. Then for any $\tLambda,\tLambda'\in\uLeg(\Lambda_*)$
    \[\dCSosc(\tLambda,\tLambda')\leq 3 A \cdot \ud(\tLambda,\tLambda'),\]
$A=\dCSosc(\tLambda_*,\tphi_1\cdot \tLambda_*)$, $\tLambda_*\in\uLeg(\Lambda_*)$ being the constant isotopy equivalence class, $\tphi_1\in\widetilde{\conto}$ the lift of the loop $(\varphi_t)_{t\in S^1}$ and $\ud$ the distance defined in Theorem \ref{thm2}.
\end{prop}

The proof of Proposition \ref{prop : distance spectrale} repeats almost verbatim the proof of Theorem 4.19 \cite{allais2023spectral}.
\begin{proof}
    First, for any $k\in \N$ and any $\tLambda\in\uLeg(\Lambda_*)$ we have 
\begin{equation}\label{eq:oscillation}
\dCSosc(\tLambda,\tphi_k\cdot\tLambda)=\dCSosc(\tLambda_*,\tphi_k\cdot\tLambda_*)\leq\sum\limits_{j=1}^k\dCSosc(\tphi_{j-1}\cdot\tLambda_*,\tphi_j\cdot\tLambda_*)=kA, 
    \end{equation}
    where both equalities use the invariance of $\dCSosc$ and the first one uses also the central property of $\tphi_k$. Now, let $\tLambda,\tLambda'\in\uLeg(\tLambda_*)$ and $k_0\in\N$ so that $\tphi_{-k_0}\cdot \tLambda\cleq \tLambda'\cleq\tphi_{k_0}\cdot\tLambda$. By compatibility of $\dCSosc$ with the partial order we get 
\begin{equation}\label{eq:oscillation'}
        \dCSosc(\tphi_{-k_0}\cdot \tLambda,\tLambda')\leq \dCSosc(\tphi_{-k_0}\cdot\tLambda,\tphi_{k_0}\cdot\tLambda)=\dCSosc(\tLambda,\tphi_{2k_0}\cdot\tLambda)\leq 2k_0A
    \end{equation}
    thanks to the relations \eqref{eq:oscillation} above. Finally, using the triangle inequality and both relations \eqref{eq:oscillation} and \eqref{eq:oscillation'} we get $\dCSosc(\tLambda,\tLambda')\leq \dCSosc(\tLambda,\tphi_{-k_0}\cdot\tLambda)+\dCSosc(\tphi_{-k_0}\cdot \tLambda,\tLambda')\leq 3k_0A$ which concludes the proof since one can choose $k_0:=\ud(\tLambda,\tLambda')$. 
\end{proof}

\section{Towards new invariant distances on $\uGcont$}\label{se:openquestion}

As already mentioned in the introduction, for closed contact manifold $(M,\xi)$, unbounded invariant distances on $\uGcont$ have been constructed only when the underlying contact manifold $(M,\xi)$ admits a positive loop $(\psi_t)\subset\Gcont$.  Theorem \ref{thm2} gives some new insights on how one should be able to construct an unbounded invariant distance on $\uGcont$ even if $(M,\xi)$ does not admit a positive loop $(\psi_t)\subset\Gcont$. Indeed, suppose again that $\uLeg(\Lambda_*)$ is orderable, for some closed Legendrian $\Lambda_*\subset (M,\xi)$,  and that there is a positive loop of Legendrians $(\Lambda_*^t)_{t\in S^1}\subset\Leg(\Lambda_*)$ that can be extended to a loop of contactomorphisms $(\varphi_t)_{t\in S^1}\subset\Gcont$ based at the identity and consider the invariant distance $\ud : \uLeg(\Lambda_*)\times\uLeg(\Lambda_*)\to\N$ defined in Theorem \ref{thm2}. It is easy to see that the map 
\begin{equation}\label{eq : norme}\nu : \uGcont \to \N\cup\{+\infty\}\quad \quad \quad \tpsi\mapsto \sup\{\ud(\tLambda,\tpsi\cdot\tLambda)\ | \ \tLambda\in \uLeg(\Lambda_*)\}\end{equation}
is an unbounded conjugation invariant norm if and only if it takes value in $\N$. Thanks to Proposition \ref{prop : distance spectrale}, for any contact form $\alpha$ supporting $\xi$ and any $T>0$
\[\sup\{\ell_+^\alpha(\tpsi\cdot\tphi_\alpha^T\cdot \tLambda_*,\tpsi\cdot \tLambda_*)\ |\ \tpsi\in\uGcont\}<+\infty.\]
If moreover, there exists $T>0$ such that  $\inf\{\ell_-^\alpha(\tpsi\cdot\tphi_\alpha^T\cdot\tLambda_*,\tpsi\cdot\tLambda_*)\ |\ \tpsi\in\uGcont\}>0$ then the map $\nu$ does take value in $\N$ and thus would define an unbounded conjugation invariant norm.

\section{Appendix : On the universal cover}\label{se:appendix}
Recall that a topological space $X$ which is connected and locally simply connected admits a universal cover $\widetilde{X}$. Moreover, as a set, $\widetilde{X}$ can be identified with equivalence classes of continuous paths $\gamma : [0,1]\to X$ starting at a fixed point $x_0\in X$. The equivalence relation $\sim_\tau$ on continuous paths $C^0([0,1],X)$ is defined as follow: $\gamma \sim_\tau \gamma'$ if the exists a continuous map $\Gamma : [0,1]\times [0,1]\to X$ such that $\Gamma(s,0)=\gamma(0)=\gamma'(0)$, $\Gamma(s,1)=\gamma(1)=\gamma'(1)$, $\Gamma(0,t)=\gamma(t)$ and $\Gamma(1,t)=\gamma'(t)$. See \cite{johnlee} for more details. \\

In this appendix the aim is to show that for a closed and connected Legendrian $\Lambda_*$ in a cooriented contact manifold $(M,\xi)$ the space $\Leg(\Lambda_*)$ with the $C^k$-topology, for $k\in \N_{>0}\cup\{+\infty\}$, admits a universal cover in the sense described just above that coincides with $\uLeg(\Lambda_*)$ described in Section \ref{se:preliminaire}.\\

More precisely the $C^k$-topology on $\Leg(\Lambda_*)$ is by definition the topology it inherits as a subset of the quotient topological space $C^\infty(\Lambda_*,M)/\mathrm{Diff}(\Lambda_*)$, i.e. the space of smooth maps from $\Lambda_*$ to $M$, endowed with the $C^k$-topology (see \cite{Hirsch1994}), modulo composition by diffeomorphisms of $\Lambda_*$. From now on, we denote by $\Leg_k(\Lambda_*)$ the space $\Leg(\Lambda_*)$ endowed with the $C^k$-topology and by $C_{k+1}^\infty(\Lambda,\R)$ the space of smooth functions defined on a closed manifold $\Lambda$ endowed with the $C^{k+1}$-topology. \\

It is clear that $\Leg_k(\Lambda_*)$ is path connected by definition of $\Leg(\Lambda_*)$ (see Section  \ref{se: isotopies}). Moreover, any point $\Lambda\in\Leg_k(\Lambda_*)$ admits a Weinstein neighborhood, i.e. a neighborhood that is homeomorphic to a convex neighborhood of the $0$-function in $C_{k+1}^\infty(\Lambda,\R)$ (see for example \cite{Tsuboi}). Therefore $\Leg_k(\Lambda_*)$ is locally contractible. We denote by $\uLeg_k(\Lambda_*):=\{\gamma\in C^0([0,1],\Leg_k(\Lambda_*))\ |\ \gamma(0)=\Lambda_*\}/\sim_\tau$ the universal cover of $\Leg_k(\Lambda_*)$.

\begin{prop}\label{prop:universalcover}
    The sets $\uLeg_k(\Lambda_*)$ and $\uLeg(\Lambda_*)$ are canonically isomorphic, i.e. there exists a canonical bijection between $\uLeg(\Lambda_*)$ and $\uLeg_k(\Lambda_*)$. 
\end{prop}

Endowing $C^0([0,1],\Leg_k(\Lambda_*))$ with the $C^0$-topology we prove the following Lemma that will help proving Proposition \ref{prop:universalcover}.

\begin{lem}\label{lem:universalcover}
    For any $\gamma\in C^0([0,1],\Leg_k(\Lambda_*))$ and any non-empty neighborhood $U$ of $\gamma$ there exists $\gamma'\in U$ such that $(\gamma'(t))$ is a Legendrian isotopy and $ \gamma'\sim_\tau \gamma$. 
\end{lem}
\begin{proof}
    Let $N\in\N$ be big enough so that there exist Weinstein neighborhoods $\mathcal{U}_i\subset\Leg_k(\Lambda_*)$ of $\gamma(\frac{i}{N})$ for all $i\in[0,N-1]\cap\N$ satisfying:
    
    \begin{enumerate}[1.]
        \item $(\gamma_i(t)):=(\gamma(\frac{i+t}{N}))_{t\in[0,1]}$ is contained in $\mathcal{U}_i$ 
        \item If a path $\eta\in C^0([0,1],\Leg_k(\Lambda_*))$ is such that $(\eta(t))_{t\in[\frac{i}{N},\frac{i+1}{N}]}\subset \mathcal{U}_i$ for all $i\in[0,N-1]\cap\N$ then $\eta\in U$. 
     \end{enumerate}
   Let us denote by $\Phi_i : \mathcal{U}_i\simeq \mathcal{V}\subset  C_{k+1}^\infty(\gamma(\frac{i}{n}),\R)$ the homeomorphism, where $\mathcal{V}$ is a convex $C^{k+1}$-neighborhood of the $0$-function and $\Phi_i(\gamma(\frac{i}{n}))=0$. Consider the map $\gamma_i'\in C^0([0,1],\Leg_k(\Lambda_*))$ defined as $t\mapsto \Phi_i^{-1}(t f)$ where $f:=\Phi_i(\gamma(\frac{i+1}{N}))$. It is easy to check that the concanated path\footnote{concatenation of paths is defined similarly as for Legendrian isotopies \eqref{eq:concatenation}} $\gamma':=\gamma_0'*\cdots *\gamma_{N-1}'$ induces the desired Legendrian isotopy $(\gamma'(t))$. 
\end{proof}

\begin{proof}[Proof of Proposition \ref{prop:universalcover}]
Let $\mathcal{P}(\Leg)$ be the space of all Legendrian isotopies $(\Lambda_t)\subset \Leg(\Lambda_*)$ (not necessarily starting at $\Lambda_*$). The map $\Psi : \mathcal{P}(\Leg)\to C^0([0,1],\Leg_k(\Lambda_*))$ that associates to a Legendrian isotopy $(\Lambda_t)$ the path $t\mapsto \varphi(t,\Lambda_0)$ where $\varphi : [0,1]\times \Lambda_0\to M$ is a smooth parametrization of $(\Lambda_t)$ is well defined. Its restriction to $\mathcal{P}(\Lambda_*)$, the set of Legendrian isotopies starting at $\Lambda_*$, descends to a map $\Phi : \uLeg(\Lambda_*)\to \uLeg_k(\Lambda_*)$ which is surjective thanks to Lemma \ref{lem:universalcover}.  Let us prove that $\Phi$ is also injective. To do so,  consider two Legendrian isotopies $(\Lambda_t),(\Lambda_t')\in\mathcal{P}(\Lambda_*)$ that satisfy $\Psi(\Lambda_t)\sim_\tau \Psi(\Lambda'_t)$. The aim is to show that $(\Lambda_t)\sim(\Lambda_t')$. Let $\gamma:=\Psi(\Lambda_t)$, $\gamma':=\Psi(\Lambda_t')$ and $\Gamma :[0,1]\times [0,1]\to \Leg_k(\Lambda_*)$ a continuous homotopy, relative to endpoints, between $\gamma$ and $\gamma'$. Let $N\in\N$ be big enough so that $\left(\gamma_{i,j}(s,t):=\Gamma(\frac{i+s}{N},\frac{j+t}{N})\right)_{(s,t)\in [0,1]^2}$ is contained in a Weinstein neighborhood $\mathcal{U}_{i,j}$ of $\Gamma(\frac{i}{n},\frac{j}{n})$ for all $i,j\in [0,N-1]\cap \N$.

\textbf{First:} Remark that $(\gamma_{i,j}(0,t))_{t\in [0,1]}$ is contained in $\mathcal{U}_{i-1,j}\cap\mathcal{U}_{i,j}$ for any $i\in[1,N-1]\cap \N$ and $j\in[0,N-1]\cap\N$. By Lemma \ref{lem:universalcover}, there exists a Legendrian isotopy $(\Lambda^{i,j}_t)$ contained in $\mathcal{U}_{i-1,j}\cap\mathcal{U}_{i,j}$ so that $\Psi(\Lambda^{i,j}_t)\sim_\tau\gamma_{i,j}(0,\cdot)$.
The concatenation $(\Lambda_t^i):=(\Lambda_t^{i,0})*\cdots *(\Lambda_t^{i,N})$ is a Legendrian isotopy for all $i$ and $\Psi(\Lambda_t^i)\sim_\tau \Gamma(\frac{i}{n},\cdot)$. 

    \textbf{Second:} Using the transitivity of $\sim$, to conclude, it remains to prove that $(\Lambda_t^{i})\sim (\Lambda_t^{i+1})$ for any $i\in[0,N-1]$, where by convention $(\Lambda_t^0):=(\Lambda_t)$ and $(\Lambda_t^N):=(\Lambda_t')$ are  our inital isotopies.
    By construction, for any $i,j\in[0,N-1]\cap\N$, the Legendrian isotopies $(\Lambda_t^{i,j})$ and $(\Lambda_t^{i+1,j})$ are contained in the Weinstein neighborhood $\mathcal{U}_{i,j}$. Denote by $\Phi_{i,j}$ the homeomorphism between $\mathcal{U}_{i,j}$ and a convex neighborhood $\mathcal{V}$ of the $0$-function in $C^\infty(\Gamma(\frac{i}{N},\frac{j}{N}),\R)$ and consider the two corresponding isotopies $\left(f_t:=\Phi_{i,j}(\Lambda_t^{i,j})\right)$, $\left(g_t:=\Phi_{i,j}(\Lambda_t^{i+1,j})\right)\subset \mathcal{V}$. Finally we define the $[0,1]^2$ family of Legendrians $(\Lambda^{i,j}_{s,t})_{(s,t)\in[0,1]^2}$ to be  $(\Phi_{i,j}^{-1}(s\cdot f_t+(1-s)\cdot g_t))_{(s,t)\in[0,1]^2}$.  It follows that $(\Lambda_{s,t}^i)_{t\in [0,1]}:=(\Lambda_{s,t}^{i,0})_{t\in I}*\cdots*(\Lambda_{s,t}^{i,N-1})_{t\in [0,1]}$, by varying $s$ between $[0,1]$, gives the desired isotopy between $(\Lambda_t^i)$ and $(\Lambda_t^{i+1})$ which concludes the proof. 
\end{proof}

\bibliographystyle{amsplain}
\bibliography{biblio} 

\end{document}